\documentclass[11pt,twoside]{amsart}
\usepackage{amsmath, amsthm, amsfonts, amssymb}
\usepackage[mathcal,mathscr]{eucal}
\usepackage{mathrsfs}
\usepackage{graphicx, xcolor}
\usepackage[colorlinks]{hyperref}

\DeclareMathOperator{\Aut}{Aut}
\DeclareMathOperator{\Core}{Core}
\DeclareMathOperator{\Sol}{Sol}
\DeclareMathOperator{\Syl}{Syl}

\DeclareMathOperator{\PGL}{PGL}
\DeclareMathOperator{\PSL}{PSL}
\DeclareMathOperator{\SL}{SL}
\DeclareMathOperator{\Fit}{Fit}
\newcommand{\cyc}[1]{\langle #1\rangle}
\newcommand{\C}{\mathcal {C}}
\newcommand{\N}{\mathcal{N}}
\newcommand{\sub}{\leqslant}
\newcommand{\nsub}{\nleqslant}

\newcommand{\nor}{\trianglelefteq}

\newcommand{\bk}{\backslash}

\newtheorem{theorem}{Theorem}[section]
\newtheorem{lemma}[theorem]{Lemma}

\newtheorem{corollary}[theorem]{Corollary}

\newtheorem*{problem}{Problem}
\newtheorem{remark}[theorem]{Remark}
\newtheorem{conjecture}{Conjecture}

\begin{document}

\title[\scriptsize{ The impact of the solubilizer of an element...}]{The impact of the solubilizer of an element on the structure of a finite group}

\author[ Hamid Mousavi , Mina Poozesh and Yousef Zamani]{Hamid Mousavi, Mina Poozesh and Yousef Zamani$^{\ast }$}

\address{Department of Pure Mathematics, Faculty of Mathematical Sciences, University of Tabriz, Tabriz, Iran}
\email{hmousavi@tabrizu.ac.ir}

\address{Department of Mathematics, Faculty of Basic Sciences, Sahand University of Technology, Tabriz, Iran}
\email{mi\_poozesh@sut.ac.ir}

\address{Department of Mathematics, Faculty of Basic Sciences, Sahand University of Technology, Tabriz, Iran}
\email{zamani@sut.ac.ir}
\thanks{$^{\ast}$ Corresponding author}

\maketitle

\begin{abstract}
Let $G$ be a finite group, and let $x$ be an element of $G$.
Denote by $\Sol_G(x)$ the set of all $y \in G$ such that the group generated by $x$
and $y$ is soluble. We investigate the influence of $\Sol_G(x)$ on the structure of $G$.
\end{abstract}
\section{\bf Introduction}
Some of the properties of a finite group $G$ are determined by the properties
of its $2$-generated subgroups. For instance, according to Zorn's theorem in \cite{Zo},
a finite group $G$ is nilpotent if and only if every two-generated subgroup of $G$ is nilpotent.
In \cite{Ba2}, Baer proved that $G$ is supersoluble if and only if every two-generated subgroup of $G$ is supersoluble.
A similar result for the solubility property is provided by John G. Thompson in \cite[Corollary 2]{Thompson-1}
which is an outstanding criterion for the solubility of finite groups. The result reads as follows:

{\em A finite group $G$ is soluble if and only if every two-generated subgroup of $G$ is soluble.}

Given a group $G$, we say that $g\in G$ is a \emph{radical element} if for every $x\in G$, the subgroup
generated by $x$ and $g$ is soluble. For the soluble radical of $G$, which is denoted by $ R(G)$,
the following extension of Thompson's theorem has been proved in \cite{Guralnick}:

{\em The soluble radical $R(G)$ of $G$ coincides with the collection of all radical elements in $G$.}

An interesting topic is the characterization of finite groups in terms of soluble
two-generated subgroups. To achieve this goal, for an element $x\in G$ we define the {\em solubilizer} of $x$ in $G$ by
$$
\Sol_G (x):= \{ g \in G ~|~ \langle x, g \rangle \ \mbox{is soluble} \}.
$$
It might be just a subset of $G$ and not a subgroup.
Note that $\Sol_G (x) = G$ if and only if $x$ is a radical element in $G$ and equivalently $x\in R(G)$.

In \cite[Theorem 3.1]{Akbari}, it is proved that, if for some $x\in G$, the elements of $\Sol_G (x)$ commute pairwise, then $G$ is abelian. In \cite[Lemma 3.1]{Akbari1}, extending this result it is proved that if for every $u_1, \dots, u_k \in \Sol_G(x)$ we have the left-normed commutator $[u_1,\dots, u_k] = 1$, then $\Sol_G(x)$ is a subgroup. Moreover it is proved that $\Sol_G(x)$ is nilpotent of class at most $k-1$, and that $G$ is nilpotent
of nilpotency class at most $2$ if and only if $k=3$. Also, it is raised the question whether we always have the $k$-th term of the lower central series of G, $\gamma_k(G) = 1$ in the case that for every $u_1, \dots, u_k\in\Sol_G(x)$ it is $[u_1, \dots, u_k] = 1$. 
 
 Our answer to this question is no. For a counter example assume that $G\cong\PGL(2,7)$, then by using GAP~\cite{GAP}, we see that for all $x\in G$ of order $8$, $\Sol_G(x)\cong D_{16}$ is a Sylow $2$-subgroup of  $G$.
 
It is clear that if $\Sol_G(x)$ is a $2$-subgroup of $G$ for some $x \in G$, it must be a Sylow $2$-subgroup of $G$. Therefore, its structure is expected to be a crucial factor in determining the structure of the insoluble group $G$. The question now arises: in which insoluble groups $G$ is $\Sol_G(x)$ a subgroup?

In this article, we answer this question in the specific case where $\Sol_G(x)$ is a maximal and meta-cyclic $2$-subgroup of $G$ for some $x$ in $G$. Additionally, we show that if $|\Sol_G(x)|=2p$ for an odd prime number $p$, then $G$ is a simple group.
Some other results are also given.

\section{\bf  Preliminaries and some properties of solubilizers}
 Now we state some elementary results about the solubilizer of an element in a finite group.
 \begin{theorem}[Deskins, Janko, Huppert]\cite[Satz 4.7.4]{Huppert}\label{Deskins}
 Assume that $G$ has a nilpotent maximal subgroup $M$, such that the nilpotency class of  a Sylow $2$-subgroup of $M$ is at most $2$. Then $G$ is soluble.
\end{theorem}   
 \begin{remark}\label{R1}
 Assume that  $G$ is an  insoluble group and $\Sol_G(x)$ is a subgroup of $G$. Then, by the above theorem, for any $x\in G$,  a Sylow $2$-subgroup of $\Sol_G(x)$ is non-abelian of order at least $16$, otherwise, if $M\sub G$ such that $\Sol_G(x)$ is a maximal in $M$, then $M$ must be soluble, a contradiction, because $\Sol_G(x)$ is the maximal soluble subgroup of $G$ containing $x$.
  \end{remark}
 
\begin{lemma}\label{proper-sol}\cite{Akbari, Hai}
Let $G$ be a finite group and $N\nor G$. Then, the following statements hold for any $x\in G$.
\begin{itemize}
\item[(1)]
$
\langle x\rangle \subseteq \mathcal{N}_{G}(\langle x\rangle)\cup R(G) \subseteq \Sol_G (x)=\bigcup_{H}H,\\
$
where the union ranges over all soluble subgroups $H$ of $G$ containing $x$.
\item[(2)] $|x|$ divides $|\Sol_G (x)|$.
\item[(3)] $\frac{\Sol_G (x)N}{N} \subseteq \Sol_{G/N} (xN) ~($obviously, here  $\frac{\Sol_G (x)N}{N}:=\{ yN~\big{|}~y\in \Sol_G(x)\})$.
\item[(4)] If $N$ is soluble, then $\Sol_{\frac{G}{N}} (xN)=\frac{\Sol_G (x)}{N}$.
\item[(5)] If $N$ is soluble, then $|\Sol_G (x)|$ is divisible by $|N|$. In particular, $|\Sol_G (x)|$ is divisible by $|R(G)|$. Furthermore,
 $|\frac{\Sol_G (x)}{N}|=\frac{|\Sol_G (x)|}{|N|}$.
 \item[(6)] If $G$ is insoluble, then $\langle x\rangle $ is properly contained in $\Sol_G (x)$.
 \item[(7)] $|\Sol_G (x)|$ cannot be equal to a prime.
 \item[(8)] If $\langle x\rangle =\langle y\rangle $ for some $y\in G$, then $\Sol_G (x)=\Sol_G (y)$.
 \item[(9)] For any $g\in G$, $\Sol_G (x^g)=\Sol_G (x)^g$.
\item[(10)] For every $x\in G$, $|\mathcal{C}_{G}(x)|$ divides $|\Sol_G(x)|$.
\item[(11)] A finite group $G$ is soluble if and only if $\Sol_G (x)$ is a subgroup of $G$ for all $x\in G$.
  \end{itemize}
 \end{lemma}

Similar to Lemma~\ref{proper-sol}-(i), $\N_G(H)\subseteq  \Sol_G (x)$  for any soluble subgroup $H$ of $G$ when  $x\in H$, because $\cyc{x,y}\sub\cyc{H,y}$ is soluble for any $y\in\N_G(H)$.

\begin{lemma}\label{product}
Let $H$ and $K$ be two subgroups of $G$ such that $G=HK$ and $[H,K]=1$. Assume that $x\in H$, then $\Sol_G(x)=K\Sol_H(x)$. 
\end{lemma}
\begin{proof}
Obviously $K\Sol_H(x)\subseteq\Sol_G(x)$. Assume that $g\in\Sol_G(x)$, then $g=hk$ for some $h\in H$ and $k\in K$ and $\cyc{x,hk}$ is soluble. Since $\cyc{x,hk}\sub\C_G(k)$, $\cyc{x,h}\sub\cyc{x,hk,k}$ is soluble. Therefore $h\in\Sol_H(x)$ and $g\in K\Sol_H(x)$. The assertion is obtained.
\end{proof}

\section{$\Sol$ as a subgroup of $G$}
While $\Sol_G(x)$ may not always be a subgroup for a finite group $G$, it is important to note that for an insoluble group $G$, where $G = \cyc{x^G}$ for some involution $x$ of $G$, $\Sol_G(x)$ is definitely not a subgroup of $G$.

 \begin{lemma}\label{2-power}
Let $G$ be a finite  group and $x$ be an arbitrary involution of $G$. If $\Sol_G(x)$ is a subgroup of $G$, then $x^G\subset\Core_G(\Sol_{G}(x))$. If  in additional $G=\cyc{x^G}$, then $G$ is soluble.
\end{lemma}
\begin{proof}
Since for any $g\in G$, the group $\cyc{x, x^g}$ is either cyclic or dihedral, i.e. in particular soluble, we have $x\in \Sol_G(x^g)$.  As
$$
\Core_G(\Sol_G(x))=\bigcap_{g\in G}\Sol_G(x)^g=\bigcap_{g\in G}\Sol_G(x^g),
$$
 we conclude that $x\in \Core_G(\Sol_G(x))$. Therefore $x^G\subset\Core_G(\Sol_{G}(x))$.

Now assume that $G=\cyc{x^G}$. Then $G=\Sol_G(x)$, thus $x\in R(G)$ and so $G=\cyc{x^G}=R(G)$ is soluble.
\end{proof}
By the above lemma, if $G$ is a simple group and $x \in G$ is an involution, $\Sol_G(x)$ is not a subgroup of $G$. The same holds if $G$ is a symmetric group of degree at least 5.

\begin{lemma}\label{lem self}
Finite insoluble groups do not have self-normalizing subgroups of prime order.  	
\end{lemma}
\begin{proof}
Let $G$ be a finite insoluble group and $\langle x \rangle$ be  a self-normalizing subgroup of order $p$, where $p$ is a prime number. Then $\mathcal{C}_G(x)=\langle x \rangle$, hence the {\rm Sylow} $p$-subgroup of $G$ is of order $p$. Suppose $P$ is a Sylow $p$-subgroup of $G$.
Then $\mathcal{N}_G(P)=\mathcal{C}_G(P)$ and  so $P$ has a normal complement in $G$, say $N$. Since $P=\mathcal{C}_G(P)$, so $G=NP$ is a Frobenius group with kernel $N$, which is a contradiction.
\end{proof}

\begin{lemma}\label{L.1}
Let $G$ be a finite group and $x\in G$. Then either $\N_G(\cyc{x})=\Sol_G(x)$ or $|\Sol_G(x)|> \ell|x|$, where $\ell=\min \{|\cyc{x}: \cyc{x}\cap\cyc{x^y}|\,|\, y\not\in\N_G(\cyc{x})\}$.
\end{lemma}
\begin{proof}
Assume that $\N_G(\cyc{x})\neq\Sol_G(x)$ and  $y\in\Sol_G(x)\bk\N_G(\cyc{x})$. Then $\cyc{x}\cyc{x^y}\subset\cyc{x,y}$, otherwise $y\in\cyc{x}$. Hence $\ell|x|<|\cyc{x,y}|\leq|\Sol_G(x)|$.
\end{proof}

Assume that $x$ is of prime order $p$. In this case $\ell=|x|$, so either $\N_G(\cyc{x})=\Sol_G(x)$ or $|\Sol_G(x)|> p^2$. Now immediately we  conclude the second main result and  Proposition 4.8 of \cite{Akbari1}.
\begin{corollary}\cite[Theorem B]{Akbari1}
Let $G$ be a finite insoluble group and $x$ an element of $G$. Then
$|\Sol_G(x)|\neq p^2$ for any prime $p$.
\end{corollary}

\begin{corollary}\label{C.1}
Let $G$ be a finite group. Suppose that $P\in\Syl_p(G)$ for some prime $p$ and $x\in P$ such that $|x|=\exp(P)$. Then either $\N_G(\cyc{x})=\Sol_G(x)$ or $|\Sol_G(x)|> p\exp(P)$.
\end{corollary}

Now we have the following structural theorem.

\begin{theorem}\label{2p}
Let $G$ be a finite insoluble group and for some $x\in G$, $|\Sol_G(x)|=2p$, where $p$ is an odd prime number. Then $G$ is simple and $\N_G(\cyc{x})=\Sol_G(x)$.	
\end{theorem}
\begin{proof}
By Lemma \ref{proper-sol}-(2), (10), we have $\cyc{x}=C_G(x)$ is of prime order (otherwise, $\Sol_G(x)=C_G(x)$ is abelian, a contradiction). Therefore $\cyc{x}$ is a Sylow subgroup of $G$ and $|x|=p$, because $G$ is insoluble. Also by Lemma~\ref{L.1}, $\Sol_G(x)=\N_G(\cyc{x})$.

Assume that $G$ is not simple and $N$ is a  minimal normal subgroup of $G$. If $p\mid\,|N|$, then $G=N\N_G(P)$. So
 $\N_N(P)=P$ and by Lemma \ref{lem self}, $N$ is soluble, which contradicts the insolubility of $G$. Therefore $p\nmid|N|$, so
 $P\cap N={1}$ and $PN$ is a Frobenius group with kernel $N$ (because $P=\C_G(P)$). Hence $N$ is nilpotent and $NP$ is soluble. It implies that $NP\leq \Sol_G(x)=\N_G(P)$, so $|N|=2$ and $\Sol_G(x)$ is abelian, a contradiction.
\end{proof}
\begin{remark}\label{pq}
Similar to the proof of Theorem~\ref{2p}, we can see that if $|\Sol_G(x)|=pq$, where $|x|=q>p$ are primes, then $G$ is simple and  $N_G(\cyc{x})=\Sol_G(x)$.
\end{remark}

By GAP \cite{GAP} we get $|\Sol_{S_{7}}((1~2)(3~4))|=2\cdot 3 \cdot 7$, $|\Sol_{S_{5}}((1~2~3)(4~5))|=2^2\cdot 3$ and $|\Sol_G(x)|=2\cdot 3 \cdot 7$, where $G=\PSL(2,11)$ and $x$ is an element of order $3$. These examples show that if the $|\Sol_G(x)|$  for some $x\in G$ is the product of more than two prime numbers, then G is not necessarily simple.

Again by GAP, we observe that $\Sol_{A_5}(x)\cong D_{10}$, for some $x$ of order $5$ and $\Sol_{\PSL(2,7)}(x)\cong C_7\rtimes C_3$. Now the following problem arises naturally.

\begin{problem}
Suppose that $p$ and $q$ are two distinct primes numbers. Find the structure of finite simple groups $G$ such that,  $|\Sol_G(x)|=pq$ for some $x\in G$.
\end{problem}

In \cite[Lemma 4.2]{Akbari1}, the authors show that (with a long proof) $|\Sol_G(x)|\neq 8$ for all insoluble group $G$ and $x\in G$. In the following  we give a short proof of the this lemma.

Let $Q\in\Syl_2(G)$. Since $Q$ is not cyclic and by Theorem~\ref{Deskins}, $\Sol_G(x)$ is not a subgroup of $G$, so  $2< |Q|\leq |\N_G(Q)|< 8$. Therefore $Q=\N_G(Q)$ is of order $4$. Thus $G$ is $2$-nilpotent, a contradiction.

\begin{theorem}\label{T.2}
Let $G$ be a finite insoluble group and $x\in G$. If $|\Sol_G(x)|=16$, then  $\Sol_G(x)\sub G$.
\end{theorem}
\begin{proof}
Suppose that $\Sol_G(x)$ is not a subgroup of $G$ and $Q$ is a Sylow $2$-subgroup of $G$ containing $x$. Then  $4\leq |Q|\leq 8$ and $|\mathcal{N}_G(Q)|< 16$.  Also $R(G)=1$, otherwise $|\Sol_{G/R(G)}(xR(G))|$ is a power of $2$ less than or equal to $8$, a contradiction. Now, the following two cases can be distinguished.

Case 1: $|Q|=8$.

Assume that $M$ is a subgroup of $G$ such that  $Q$ is a maximal subgroup of $M$. Then $M$ is soluble by Theorem~\ref{Deskins}. Since $|Q|<|M|\leq |\Sol_G(x)|$, thus $M=\Sol_G(x)$, which is a contradiction.

Case 2: $|Q|=4$

Since $Q$ is not a self-normalizing subgroup of $G$, $|\mathcal{N}_G(Q)|=12$ and $\mathcal{C}_G(Q)=Q$ is elementary abelian. Assume that $G$ is not simple and $N$ is a normal subgroup of $G$. Then $2\mid |N|$ because $R(G)=1$. If $|Q\cap N|=2$ then $N$ is $2$-nilpotent and normal $2$-complement of $N$ is characteristic of odd order, which is a contradiction. Thus $Q\leqslant N$ and $G=N\mathcal{N}_G(Q)$, by Frattini argument. Therefore $\mathcal{N}_N(Q)=Q=\mathcal{C}_N(Q)$, again $N$ is $2$-nilpotent, which is a contradiction. 

Thus $G$ is simple and  $Q=\mathcal{C}_G(x)$, because $|\mathcal{C}_G(x)|\mid |\Sol_G(x)|$. Now by \cite[Theorem 15.2.5]{Gorenstein}, $G\cong A_5$. Since all involutions in $A_5$ are conjugate so $|\Sol_G(x)|=36$ by using GAP, which is final contradiction.
\end{proof}

\section{$\Sol$ as a $2$-subgroup}
%\begin{remark}\label{nil}
Let $G$ be a finite insoluble group and $x$ be an element of $G$. By Remark~\ref{R1}, $\Sol_G(x)$ cannot admits the structure of a $p$-group, where $p$ is an odd prime number. Also by Lemma~\ref{L.1}, $|\Sol_G(x)|\neq p^n$ if $|x|=p^{n-1}$ where $p$ is odd.  
 Now in the following we show that, for some $x\in G$, $\Sol_G(x)$ can be a $2$-subgroup of $G$ of order grater than $8$. Also $\Sol_G(x)$ can be of size $2^n$ when $|x|=2^{n-1}$.

\begin{lemma}\label{psl}
Let  $G$ be a finite insoluble group and $\Sol_G(x)$ be a $2$-subgroup of $G$ for some $x\in G$. Then $\Sol_G(x)$ is a Sylow $2$-subgroup of $G$. In addition, if $G/R(G)$ is isomorphic to direct product of $\ell$ copies of $\PSL(2,p)$, where $p$ is prime, then $p\geq 31$ is a Mersenne prime number, $|x|\geq 8$ and $|\Sol_G(x)|=\ell(p+1)|R(G)|$.
\end{lemma}
\begin{proof}
By Lemma~\ref{proper-sol}-(2), $x$ is a $2$-power order too.  Let $Q$ be a suitable Sylow $2$-subgroup of $G$ such that $x\in Q=\Sol_G(x)$.

Set $\bar{G}=G/R(G)$ and $\bar{x}=xR(G)$. As $R(G)\sub Q$, $\Sol_{\bar{G}}(\bar{x})$ is of $2$-power order. Assume that $G=G_1\cdots G_{\ell}$, where $G_i\cong\PSL(2,p)$ and $[G_i,G_j]=1$ for any $1\leq i<j\leq\ell$. By Lemma~\ref{product}, $\bar{x}\not\in G_i$, for all $i\leq\ell$. Let $\bar{x}=x_1\cdots x_{\ell}$, where $x_i\in G_i$ is an $2$-element.

A Sylow $2$-subgroup of $G_i$ is isomorphic to either $C_2\times C_2$, where $p\equiv \pm 3\pmod{8}$ or  a self-normalizing maximal subgroup of dihedral type, where $p\geq 17$ is a Fermat or Mersenne prime number. 

In the first case, we can assume that $Q_i=\cyc{x_i,y_i}$ is Sylow $2$-subgroup of $G_i$ and $\N_{G_i}(Q_i)=\cyc{x_i,y_i}\rtimes\cyc{s_i}$, where $|s_i|=3$. Now $\cyc{\bar{x},\bar{y}}\rtimes\cyc{\bar{s}}\cong A_4$, where $\bar{y}=y_1\cdots y_{\ell}$ and $\bar{s}=s_1\cdots s_{\ell}$. Therefor $3\mid\,
|\Sol_{\bar{G}}(\bar{x})|\mid\,|\Sol_G(x)|$, a contradiction. 

In the second case if  $p$ is a Fermat prime number, by \cite[Theorem 6.25]{Suzuki}, $G_i$ has a maximal subgroup $M_i=\cyc{s_i,y_i}\cong C_p\rtimes C_{(p-1)/2}$, where $|s_i|=p$ and $y_i$ is an $2$-element of order $(p-1)/2$. Also $G_i$ contains a subgroup $D_i=\cyc{s_i,y_i}$ isomorphic to $D_{2p}$, where $|s_i|=p$ and $|y_i|=2$. If for some $i$, $\cyc{x_i}$ is a maximal cyclic subgroup of order $2$, then $x_i\in D_i$, otherwise we can assume that $x_i\in M_i$ for all $i$. So in any case $\cyc{s_i,x_i}$ is soluble.
 Since $\cyc{s_i,x_1\cdots x_{\ell}}\sub \cyc{s_i, x_i,\prod_{i\neq j}x_j}$, so it is soluble, thus $p\mid|\Sol_{\bar{G}}(\bar{x})|$, a contradiction.
Thus $p$ is a Mersenne prime number.

If for some $i$, $|x_i|\leq 4$, then $G_i$ has a maximal subgroup $M_i$ isomorphic to $S_4$ such that $x_i\in M_i$. Hence for some $i$, $\cyc{x_i, s_i}$ is soluble where $|s_i|=3$. Thus $\cyc{s_i, x_i\cdots x_{\ell}}$ is soluble and so $3\mid|\Sol_{\bar{G}}(\bar{x})|$, thus $p\neq 7$. Therefore $p\geq 31$
and $$\Sol_{G}(x)/R(G)=Q/R(G)\cong\underset{\ell}{\underbrace{ D_{p+1}\times\cdots\times D_{p+1}}}.$$
\end{proof}

\begin{corollary}\label{2-power_1}
Let $G$ be a finite minimal insoluble group and $\Sol_G(x)$  a $2$-subgroup of $G$  for some $x\in G$. Then $G/R(G)\cong\PSL(2,p)$,  where $p>31$ is  Mersenne prime and $p\not\equiv \pm 1\pmod{5}$.
\end{corollary}
\begin{proof}
Assume that  $\Sol_G(x)$ is a $2$-subgroup of $G$. By Lemma~\ref{psl}, $Q:=\Sol_G(x)$ is a Sylow $2$-subgroup of $G$ and $R(G)=\Fit(G)$. Since $G$ is minimal insoluble, $\Phi(G)=\Fit(G)$, also $Q$ is a maximal subgroup of $G$. Hence $G/R(G)$ is a minimal simple group with nilpotent maximal subgroup. According to \cite[Main Theorem]{Baumann} and \cite[Theorem 2]{GorWal},  $G/R(G)\cong \PSL(2,p)$ with dihedral Sylow $2$-subgroups. By Lemma~\ref{psl}, $p$ is a  Mersenne prime number, where $p\geq 31$. If $p\equiv \pm 1\pmod{5}$, then $\PSL(2,p)$ contains $A_5$ as its maximal subgroup, a contradiction, in particular $p> 31$.
\end{proof}

\begin{theorem}\label{2-power}
Let $G$ be a finite insoluble group and $\Sol_G(x)$ be a $2$-subgroup of $G$ for some $x\in G$.  If $\Sol_G(x)$ is a maximal subgroup of $G$, then 
 $N\leqslant G/R(G)\leqslant \Aut(N)$, where $N$ is a direct 
product of copies of $\PSL(2,p)$, where $p \geq 31$ is a Mersenne prime number.
\end{theorem}
\begin{proof}
By Lemma~\ref{psl},  $Q:=\Sol_G(x)$ is Sylow $2$-subgroup of $G$, also $\Fit(G)= R(G)\leqslant Q$. 

(i) First suppose that  $R(G)=1$. 
As $\Fit(G)= R(G)=1$, by \cite[Main Theorem]{Baumann} and \cite[Theorem 2]{GorWal}, there exists a unique minimal normal subgroup $N$ which is a direct product of copies of $\PSL(2,p)$, where $p \geq 17$ is a Mersenne or Fermat prime number. Since $N\cap Q\unlhd Q$ and $Q$ is a maximal subgroup of $G$, so $\mathcal{N}_G(N\cap Q)=Q$. As $\C_G(N)\cap N=1$ and $G/N$ is a 2-group, then $\C_G(N)\sub Q$ and so $\C_G(N)\sub R(G)=1$. Therefore 
$$N\leqslant G\leqslant \Aut(N).$$

If $x\not\in N$, since $G/N$ is $2$-group, so $N$ contains  all Sylow $p$-subgroups of $G$ of odd order. Therefore by Frattini argument, $G=N\N_G(P)$, where $P$ is a Sylow $p$-subgroup of odd order. Hence $G/N\cong \N_G(P)/\N_N(P)$. Therefore $x\in \N_G(P)$ modulo $\N_N(P)$, thus $\cyc{P, x}$ is soluble modulo  soluble group $\N_N(P)$, which implies that $\cyc{P, x}$ is soluble and so $p\mid |\Sol_G(x)$, a contradiction. Thus $x\in N$. As $Sol_N(x)=N\cap\Sol_G(x)$ is $2$-subgroup, so by Lemma~\ref{psl}, $p\geq 31$ is a Mersenne prime number and $|x|\geq 8$. 
Since $R(G/R(G))=1$, the proof  is complete.
\end{proof}

\begin{theorem}\label{metacyclic}
Let $G$ be a finite insoluble group and $\Sol_G(x)$ a  maximal and meta-cyclic $2$-subgroup of $G$ for some $x\in G$.
\begin{itemize}
\item[(i)] If $G$ is simple, then $G\cong\PSL(2,p)$, where $p\geq 31$ is a Mersenne prime.
\item[(ii)] If  $G$ contains a proper non-abelian minimal normal subgroup $N$, then $\Sol_G(x)$ is  either dihedral or semi-dihedral type.
\begin{itemize}
\item[(ii-1)] If $\Sol_G(x)$ is of dihedral type, then $G\cong\PGL(2, p)$ and $N\cong\PSL(2,p)$, where $p\geq 7$ is a Mersenne prime number.
\item[(ii-2)] If $\Sol_G(x)$ is of semi-dihedral type, then $G\cong H(9)$ the non-split extension of $N\cong\PSL(2,9)$ by $C_2$ (with IdGroup:=(720, 765) of GAP library, for a description of $H(q)$ see \cite[Page 4]{Wong}).
\end{itemize}
\item[(iii)] If  any minimal normal subgroup of $G$ is abelian, then $G\cong\SL(2,p)$ or $\SL(2,p)\rtimes C_2$, where $p\geq 7$ is a Mersenne prime number.
\end{itemize}
\end{theorem}
\begin{proof}
(i) By Theorem~\ref{2-power}, $G\cong N\cong\PSL(2, p)$, where $p\geq 31$ is a Mersenne prime.

(ii) Assume that $O(G)$ denotes the largest normal subgroup of odd order in $G$.  Then by \cite[Corollary 1]{Randolph}, $\Fit(G)\sub\C_G(N)=O(G)=1$ and  $Q$ has the structure of dihedral or semi-dihedral type.  Now by \cite[Theorem 4]{Rose}, $G$ is isomorphic to one of the following groups:
 $$\PGL(2, p), \qquad H(9), \qquad \PGL(2, 9),$$ 
 where $p\geq 7$ is a Mersenne prime.
 By using GAP we see that for any $x\in \PGL(2,9)$, $|\Sol_{\PGL(2,9)}(x)|$ is not $2$-power and the result holds.

(iii) By parts (i) and (ii), $G/\Fit(G)$ is isomorphic to one of the following groups: \[\PSL(2,p),\qquad \PGL(2,p),\qquad H(9).\]
 Hence Sylow $2$-Subgroup of $G/\Fit(G)$ is of dihedral or semi-dihedral type.
  
 Since $Q$ is meta-cyclic, $\Fit(G)/\Phi(\Fit(G))\cong C_2$ or $C_2\times C_2$. Now for any $p>3$, the Sylow $p$-subgroup $P$ acts trivially on $\Fit(G)/\Phi(\Fit(G))$, so $[\Fit(G),P]\sub\Phi(\Fit(G))$. Hence $P\sub\C_G(\Fit(G))$, which implies that $C:=\C_G(\Fit(G))\neq \Fit(G)$. As $G/\Fit(G)$ contains unique simple  normal subgroup $N/\Fit(G)$ such that $|G: N|\leq 2$, so $N\sub\Fit(G)C$. Therefore Sylow $2$-subgroup of $C/C\cap\Fit(G)$ is of dihedral or semi-dihedral type. If  $\Fit(G)\nleqslant C$ then
$$\frac{\Fit(G)}{C\cap\Fit(G)}\times \frac{C}{C\cap\Fit(G)}\hookrightarrow\frac{G}{C\cap\Fit(G)}.$$ So Sylow $2$-subgroup of ${G}/{C\cap\Fit(G)}$ is not meta-cyclic, a contradiction. Then $\Fit(G)\sub\C_G(\Fit(G))$ and so $N\sub\C_G(\Fit(G))\sub G$. Therefore $Z(G)\sub Z(N)=\Fit(G)$.

As $G'/G'\cap\Fit(G)\cong N/\Fit(G)$, if $\Fit(G)\nleqslant G'$, then Sylow $2$-subgroup of $G/G'\cap\Fit(G)$ is not meta-cyclic. Thus $\Fit(G)\sub G'=N$.
 
 Assume that $\Phi(G)\neq\Fit(G)$, then $G=M\Fit(G)$ for some maximal subgroup $M$ of $G$. Thus $M\cap\Fit(G)\nor G$, and $M/M\cap\Fit(G)\cong G/\Fit(G)$ has a Sylow $2$-subgroup of dihedral or semi-dihedral type. Again Sylow $2$-subgroup of $G/M\cap\Fit(G)$ is not meta-cyclic, a contradiction. Then $\Fit(G)=\Phi(G)$.  Similarly $\Fit(G)=\Fit(N)=\Phi(N)$. Now the central extension 
 \[1\to \Fit(G)\to N \to \PSL(2,q)\to 1\]
is irreducible, where $q$ is prime or $q=9$. Since $N$ is perfect, by \cite[Proposition 2.1.7(i)]{Karpilovsky}, $$|\Fit(G)|\mid |M(\PSL(2, q))|=2\,\text{ or}\, 6,$$ 
where $M(\PSL(2,q))$ is the Schur multiplier of $\PSL(2,q)$.
Therefore $\Fit(G)=Z(G)$. 
Now we consider the central extension
\[1\to \Fit(G)\to G \to G/\Fit(G)\to 1.\]
Again by \cite[Proposition 2.1.7(i)]{Karpilovsky}, $|\Fit(G)|\mid |M(G/\Fit(G))|$. Therefore $G/\Fit(G)\ncong H(9)$, for $M(H(9))\cong C_3$ by  \cite[Lemma 6]{Rose}.

As $M(\PSL(2,p))\cong M(\PGL(2,p))\cong C_2$, when $p$ is odd prime, $\Fit(G)\cong M(G/\Fit(G))$. Then $G$ is a cover of $G/\Fit(G)$.  Since $\PSL(2,p)$ has a unique cover isomorphic to $\SL(2,p)$, then $N\cong\SL(2,p)$.  Therefore, either $G/\Fit(G)\cong\PSL(2,p)$ and so $G=N\cong\SL(2,p)$ or $$G/\Fit(G)\cong\PGL(2,p)\cong\PSL(2,p)\rtimes C_2,$$ and so $G\cong\SL(2,p)\rtimes C_2$ with Sylow $2$-subgroup of semi-dihedral type.
\end{proof}

\begin{remark}\label{RQ}
In Theorem~\ref{metacyclic}, we see that
Sylow $2$-subgroup of $G\cong\PGL(2,7)$ is isomorphic to $D_{16}$, and Sylow $2$-subgroup of $G\cong H(9)$ is isomorphic to $SD_{16}$. Then for some $x\in G$, $|\Sol_G(x)|=16$.

 Now for any group $A$, $\Sol_{A\times G}(x)=A\times \Sol_G(x)$, by Lemma~\ref{product}. We know that for some $x\in\PGL(2,7)$ of order $8$, $\Sol_{\PGL(2,7)}(x)\cong D_{16}$, thus $\Sol_{C_{2^m}\times\PGL(2,7)}(x)=C_{2^m}\times D_{16}$ of order $2^{4+m}$ for each $m\geq 0$.  We conclude that  for every $n\geq 4$ there exists an insoluble group $G$ such that $\Sol_G(x)$ is a $2$-subgroup of order $2^n$, for some $x\in G$.
\end{remark}

Let $G$ be an insoluble group and $x\in G\bk R(G)$. If $\Sol_G(x)$ is of prime power order or a subgroup of $G$,  then  $\Sol_{G/R(G)}(xR(G))$ is prime power order or a subgroup of $G/R(G)$, for $\Sol_{G/R(G)}(xR(G))=\Sol_G(x)/R(G)$. 

Among the  $436$ insoluble groups of order at most $2000$ (except insoluble groups of order $1920$), there are exactly fourteen groups with trivial Fitting subgroup. These groups are listed in the following table.  
\begin{table}[h]
\begin{tabular}{|c|l|}
\hline
Structure of G & IdGroup(G) in the GAP library\\ \hline
   $A_5$ &  (60, 5)\\ \hline
   $S_5$  &  (120, 34)\\ \hline
  $\PSL(3,2)$  & (168, 42)\\ \hline
   $\PSL(3,2) : C_2$ &  (336, 208)\\ \hline
   $A_6$ &  (360, 118)\\ \hline
 $\PSL(2,8)$  &  (504, 156)\\  \hline
  $\PSL(2,11)$ &  (660, 13)\\ \hline
  $S_6$  & (720, 763)\\ \hline
   $A_6 : C_2$ &(720, 764)\\ \hline
  $H(9)=A_6 \cdot C_2$  & (720, 765)\\ \hline
  $\PSL(2,13)$ &  (1092, 25)\\ \hline
  $\PSL(2,11) : C_2$ &  (1320, 133)\\ \hline
 $(A6 \cdot C_2) : C_2$  & (1440, 5841)\\ \hline
 $\PSL(2,8) : C_3$  &  (1512, 779)\\ \hline
\end{tabular}
 \\[0.2cm]
 %\caption{Insoluble groups with trivial Fitting subgroup of order at most 2000}
\end{table}

By using GAP, we checked the correctness of the following conjectures for the groups listed in Table 1. Since $\Fit(G)\sub R(G)$, the conjectures  1 and 2 are true for all insoluble groups of order at most $2000$ (except insoluble groups of order $1920$).

As $\frac{\N_G(\cyc{x})\Fit(G)}{\Fit(G)}\sub\N_{G/\Fit(G)}(\cyc{x\Fit(G)})$ and equality occurs if and only if $\Fit(G)\sub\N_G(\cyc{x}))$. So the conjecture~3 is true too for all insoluble groups $G$ of order at most $2000$ (except insoluble groups of order $1920$), when $\Fit(G)\sub\N_G(\cyc{x})$.
For the condition $\Fit(G)\nsub\N_G(\cyc{x})$ we checked the correctness of the conjecture~3 for  all insoluble groups of order less than $2000$ (except insoluble groups of order $1920$), by using GAP-code 2.

\begin{conjecture}
 Let $G$ be a finite insoluble group. If for some $x\in G$, $|\Sol_G(x)|=2^n$ then $\Sol_G(x)$ is a subgroup of $G$. 
\end{conjecture}

\begin{conjecture}
Let $G$ be a finite insoluble group. Then for any $x\in G$, $|\Sol_G(x)|\neq p^n$, where $p$ is an odd prime and $n$ is natural number.
\end{conjecture}

\begin{conjecture}
Let $G$ be a finite insoluble group. Then for any $x\in G$,
$|\N_G(\langle x \rangle)|\mid|\Sol_G(x)|$.
\end{conjecture}
If the conjecture 3 is true, then the assumption $|x|=q$ in Remark~\ref{pq}, can be removed.
%%%%%%%%%%%%%%%%%%%%%%%%%%%%%%%%%%%%%%%%%%%%%%%%%%%%%%%%%%%%%%%%%%%%%%%%%%%%%%%%%%%%%%%%%%%%%%%%%%%%%

\end{document}